\documentclass[11pt,a4paper]{amsart}
\usepackage{a4}
\usepackage[utf8]{inputenc}
\usepackage[english]{babel}
\usepackage{enumerate}

\usepackage{amsfonts,amssymb,amsmath,amscd}
\usepackage[color=green!30]{todonotes}

\usepackage{hyperref}

\usepackage{tikz}
\usetikzlibrary{matrix}
\usepackage[all]{xy}

\newtheorem{theorem}{Theorem}[section]

\newtheorem{lemma}[theorem]{Lemma}
\newtheorem{proposition}[theorem]{Proposition}

\newtheorem{theostar}{Theorem}

\theoremstyle{definition}

\newtheorem{remark}[theorem]{Remark}

\def\CC{{\mathbb C}}

\def\ZZ{{\mathbb Z}}
\def\NN{{\mathbb N}}
\def\QQ{{\mathbb Q}}

\def\RR{{\mathbb R}}
\def\HH{{\mathbb H}}

\def\sll{\mathfrak{sl}}
\def\sl2{\sll_2\CC}

\newcommand{\mC}{\mathbb{C}}
\newcommand{\mR}{\mathbb{R}}

\newcommand{\bm}{\begin{pmatrix}}
\newcommand{\ema}{\end{pmatrix}}
\newcommand{\bsm}{\left(\begin{smallmatrix}}
\newcommand{\esm}{\end{smallmatrix}\right)}

\newcommand{\Hom}{\operatorname{Hom}}

\newcommand{\SL}{\operatorname{SL}}
\newcommand{\PSL}{\operatorname{PSL}}
\newcommand{\GL}{\operatorname{GL}}
\newcommand{\SU}{\operatorname{SU}}

\newcommand{\Vol}{\operatorname{Vol}}
\newcommand{\benu}{\begin{enumerate}}
\newcommand{\eenu}{\end{enumerate}}

\newcommand{\todeux}{\tau^{(2)}(M, \rho)}
\newcommand{\detn}{\det\nolimits}
\def\op{\operatorname}

\newcommand{\wdt}[1]{\widetilde{#1}}
\newcommand{\ovl}[1]{\overline{#1}}
\newcommand{\sslash}{\mathbin{/\mkern-6mu/}}
\newcommand{\tr}{\op{tr}}
\newcommand{\neu}{\mathcal N}

\newcommand{\rel}{\mathrm{rel}}

\newcommand{\interval}[4]{
  \ifthenelse{ \equal{#1}{o} } {\mathopen{]}} {\mathopen{[}}
  #2, #3
  \ifthenelse{ \equal{#4}{o} } {\mathclose{[}} {\mathclose{]}}
}

\newcommand{\black}{\color{black}}

\title{Twisted $L^2$-torsion on the character variety}
\author{Leo Benard and Jean Raimbault}
\address{Mathematisches Institut,
Georg-August Universit\"at,
Bunsenstrasse 3-5
D-37073 G\"ottingen,  
Allemagne}
  \email{leo.benard@mathematik.uni-goettingen.de}
 \address{Institut de Math\'ematiques de Toulouse ; UMR5219 \\ Universit\'e de Toulouse ; CNRS \\ UPS IMT, F-31062 Toulouse Cedex 9, France}
  \email{Jean.Raimbault@math.univ-toulouse.fr}
\thanks{}

\begin{document}

\begin{abstract}
We define a twisted $L^2$-torsion on the character variety of 3-manifold $M$ and study some of its properties. In the case where $M$ is hyperbolic of finite volume, we prove that the $L^2$-torsion is a real analytic function on a neighborhood of any lift of the holonomy representation. 
\end{abstract}

\maketitle


\section{Introduction}


Let $M$ be a compact oriented 3--manifold with $\mu$ toroidal boundary components, $\pi =~\pi_1(M)$ its fundamental group, and assume that its interior admits a complete hyperbolic metric. Then the riemannian volume $\Vol(M)$ given by this metric is finite and it is a topological invariant of $M$ by the Mostow--Prasad rigidity theorem. The complete hyperbolic structure lifts to (a finite set of) representations $\rho_0 : \pi \to \SL_2(\CC)$ (the \textit{holonomy representation}).

Another topological invariant of $M$ is the \emph{combinatorial} $L^2$-torsion $\tau^{(2)}(M)$. It is defined in a manner similar to the classical combinatorial torsions but using a chain complex for the universal cover with the action of the fundamental group, and the Fuglede--Kadison determinants of equivariant operators. For hyperbolic manifolds it can be computed using analytic means, by a result of Wolfgang L\"uck and Thomas Schick \cite{LS}. It follows that $$\tau^{(2)}(M) = e^{-\frac{\Vol(M)}{6\pi}}.$$
Given a linear representation of the fundamental group of $M$ we can use it to twist the chain complex and for the holonomy representation of a hyperbolic 3--manifold the result of L\"uck--Schick has been extended in the (as yet unpublished) PhD thesis of Benjamin Wa\ss erman (see \cite{Wasserman} where the corresponding result is announced). In this paper we want to check that the $L^2$-torsion $\todeux$ twisted by an $\mathrm{SL}_2(\CC)$-representation $\rho$ is well-defined (at least near the holonomy), and to study the properties of the function $\rho \mapsto \todeux$ defined on (part of) the $\mathrm{SL}_2(\CC)$-character variety. We will discuss the extension of the formula in terms of volume below. 

\subsection{Well-definedness and regularity of the $L^2$-torsion function} Such a twisted $L^2$-torsion was defined (for hyperbolic knot complements) in Weiping Li's and Weiping Zhang's paper \cite{Li_Zhang}. However, the non-triviality of this invariant was never addressed in this reference, as it is not established there that the relevant complexes are of determinant-class. We remedy to this and observe that in fact the argument used to check well-definedness also implies the following result. 

\begin{theostar} \label{theo:Analyticity}
  The twisted $L^2$-torsion function 
  \[
  \todeux \colon X(M) \to \mR_{\ge 0}
  \]
  is real-analytic on an open neighborhood\footnote{in the analytic topology} $U$  of a holonomy character $[\rho_0]$ of $M$ in its character variety $X(M)$.
\end{theostar}

This result fits within a broader program. In general, the question of defining $L^2$-torsions for twisted $L^2$-chain complexes, and studying its continuity in the twisting representation, is asked by L\"uck in \cite{Lueck_twisted} (see e.g. Problem 10.11 there). It has been answered positively by Yi Liu \cite{Liu_L2alex} in the ``abelian'' case (for coefficients coming from 1-dimensional representations). Our result gives a limited answer to this problem in a non-abelian setting; however the tools we use are too weak to be applied to the general problem of continuity.

Theorem \ref{theo:Analyticity} is local in character, that is we do not know an explicit description of the locus $U$. Our argument is not well-suited for this purpose, as we use a spectral gap condition to check the determinant-class condition (see Lemma \ref{gap}) and we establish the latter through a continuity argument. It would be interesting to prove this spectral gap (if it holds) directly for all holonomies of cone-manifold structures; ideally the theorem would be proven to hold on the whole \textit{Dehn surgery space} (see \cite{Hodgson_Kerckhoff}).
On the rest of the character variety we see no reason why there should be such a spectral gap and we see no other way to check the determinant-class condition, outside of the (dense) subset of characters with values in $\ovl\QQ$, where the determinant class holds by well-known arguments (see \cite[Chapter 13]{Lueck})

Let us say a few more words on the proof of Theorem \ref{theo:Analyticity}. As we said above the main ingredient is to establish a spectral gap for the relevant operators; we immediately deduce this (by a standard continuity argument) from the result for the holonomy representation which is essentially established in by Nicolas Bergeron and Akshay Venkatesh in \cite{BV} (under the name of ``strong acyclicity''). The analyticity of $\todeux$ on this neighbourhood then follows by a general result (Lemma \ref{analytic_det}, which seems standard but we could not find in the literature) and the regularity of the twisted Laplacians on the character variety (as operator-valued functions).


\subsection{The $L^2$-torsion function and geometric invariants} In view of the results of L\"uck--Schick and their extension it is natural to ask whether there is a relation between this kind of twisted $L^2$-torsion and the volume function $\Vol$ defined on $X(M)$ (see also \cite[p. 486]{Lott} and \cite[p. 248]{Li_Zhang}). 

Using the relation between the twisted $L^2$-torsion and the volume for the family of hyperbolic manifolds obtained by Dehn surgery on $M$ and a surgery formula, one might expect an expression of the $L^2$-torsion function as the volume function plus a sum of lengths of the lengths functions $L_i, i=1, \ldots \mu$  appearing in \cite{Neumann_Zagier} which for our purposes can be defined as follows: if $\rho$ is the holonomy of a cone-manifold structure on a Dehn filling of $M$ then $L_i(\rho)$ equals the length of the core curve of the $i$-th cusp.

Unfortunately the surgery formula only computes the $L^2$-torsion of some intermediate cover of $M$, and we do not see how to extract some information on the $L^2$-torsion of the universal cover from the latter. However, \black the L\"uck--Schick result and continuity of $\todeux$ imply that $\log\todeux = -\frac {11} {12\pi} \Vol(\rho) + o(1)$ near the holonomy representation $\rho_0$. A corollary of the analyticity of the $L^2$-torsion proved in Theorem \ref{theo:Analyticity} and of \cite[Equation (3)]{Neumann_Zagier} is the following first-order approximation: 
\begin{equation} \label{equa:main}
\log\todeux= -\frac {11} {12\pi} \Vol(\rho) + \frac {11}{24} \sum_{i=1}^\mu L_i(\rho) + O(\max L_i(\rho)^2). 
\end{equation}
We do not know a closed formula for the twisted $L^2$-torsion outside of the holonomy representation and so we are not even able to determine whether it is non-constant in a neighbourhood of the holonomy representation. Computing higher-order terms in \eqref{equa:main} would be an interesting related question.
Note that it follows from our Proposition \ref{prop:Torsion} \eqref{unitary} that $\tau^{(2)}(\varrho) = e^{-\frac 1 {3\pi} \Vol(M)}$ for unitary representations $\rho$, hence the torsion is not constant on the whole character variety in general.

\subsection{$L^2$-torsion function for non-hyperbolic manifolds}
Our definition of the twisted $L^2$-torsion is rather general, and it can be considered on the character variety of any (3-)manifold. The main point is to ensure that this invariant is not trivial (i.e. zero) one needs to check some delicate conditions to hold, namely that the complex under consideration is \emph{weakly acyclic} and of \emph{determinant class}. It is achieved in several cases in this paper. Aside of hyperbolic manifolds, we consider the case where $N$ is a Seifert fibered manifold. We prove:

\begin{proposition} \label{prop:Seifert}
  Let $N$ be a compact Seifert manifold. Let $\rho$ be an irreducible representation of $\pi_1(N)$. Then the complex $C_*(N, \rho)$ is weakly $L^2$-acyclic and of determinant class, and $\tau^{(2)}(N; \rho) = 1$. 
\end{proposition}

From Proposition \ref{prop:Seifert} we deduce a JSJ formula for the twisted $L^2$ torsion:

\begin{theostar} \label{seifert}
  Let $M$ be a compact aspherical 3--manifold and $N_1, \ldots, N_r$ be the hyperbolic components in the JSJ decomposition of $M$. Let $\rho \colon \pi_1(M) \to \SL_2(\mathbb C)$ be a representation such that for any $i$, the complex $C_*^{(2)}(N_i,  \rho\vert_{\pi_1(N_i)})$ is $L^2$-acyclic and of determinant class and that for any Seifert piece $N\subset M$, the restriction $\rho\vert_{\pi_1(N)}$ is irreducible. We have
  \[
  \tau^{(2)}(M ; \rho) = \prod_{i=1}^r \tau^{(2)}(N_i ; \rho).
  \]
  In particular if $M$ is a graph manifold then $\tau^{(2)}(M ; \rho) =~1$ for any representation $\rho$ 
  that restricts to an irreducible representation on each Seifert fibered piece.
\end{theostar}
Note that the representations satisfying the hypotheses of Theorem \ref{seifert} form a non-empty open subset in the character variety of the manifold $M,$ as follows from our Theorem \ref{theo:Analyticity}. The hypotheses in the theorem are somewhat unnecessary but give a cleaner statement:  
\begin{enumerate}
\item we do not know whether an hyperbolic piece can be not of determinant class; on the other hand we will see that Seifert pieces always are. 

\item If the representation is reducible on a Seifert piece we can extract a formula from the proof which involves further factors (see Remark \ref{remk:abel} below).

\item In principle it is possible to give a similar formula for the twisted $L^2$-torsion of any 3--manifold by decomposing it into prime manifolds, which are either elliptical (for which the torsion is a classical Franz--Reidemeister torsion), $\mathbb S^2 \times \mathbb S^1$ (for which the computation is easy) or aspherical (for which the above result applies). 
\end{enumerate}
We also relate our invariant with its abelian cousin studied in \cite{DFL} by J\'er\^ome Dubois, Stefan Friedl and Wolfgang L\"uck, and treat the case of unitary representations in Proposition \ref{prop:Torsion}.

\subsection*{Organisation of the paper}
Section \ref{sec:Character} contains preliminary facts on character varieties of hyperbolic manifolds. In Section~\ref{sec:L2} we recall the general theory of $L^2$-invariants, define the twisted $L^2$ torsion and state various technical results. In Section \ref{sec:Seifert} we prove Theorem~\ref{seifert} and Proposition~\ref{prop:Seifert}. Finally, in Section \ref{sec:FKdet} we prove Theorem \ref{theo:Analyticity}. 

\subsection*{Acknowledgments}
We thank Thomas Schick for some useful suggestions on the topic. L.B. is partially funded by the RTG 2491 "Fourier Analysis and Spectral Theory".


\section{Preliminaries on character varieties and volume function} \label{sec:Character}


In this section we collect facts about the $\SL_2(\CC)$-character variety of a hyperbolic 3--manifold. In Subsection \ref{subsec:Character} we define the character variety in general, in Subsection \ref{subsec:CharHyp} we specialize to the case of hyperbolic manifolds.

\subsection{Character varieties} \label{subsec:Character}

We refer to \cite{Sikora} for a survey of character varieties in general. If $\Gamma$ is a finitely generated group and $G$ a complex affine algebraic group the representation variety $R(\Gamma, G)$ is the set $\Hom(\Gamma, G)$, which has the structure of an affine variety (equations are given by the relations defining~$\Gamma$). The group $G$ acts by conjugation on $R(\Gamma, G)$; if $G$ is reductive then the categorical quotient $ R(\Gamma, G) \sslash G$ is a well-defined affine algebraic variety. It is called the {\em $G$-character variety} of $\Gamma$ and denoted by $X(\Gamma, G)$. For a representation $\rho \in R(\Gamma, G)$ we will denote its image in $X(\Gamma, G)$ by $[\rho]$. 

In the sequel we will only consider the case $G = \SL_2(\CC)$. The character variety $X(\Gamma, \SL_2(\CC))$ is not homeomorphic in general to the topological quotient $R(\Gamma, \SL_2(\CC)) / \SL_2(\CC)$; however there is a subset $R^*(\Gamma, \SL_2(\CC))$ of irreducible representations (a representation is irreducible if it does not admit a proper nonzero invariant subspace in $\CC^2$). This is a Zariski-open subset of $R(\Gamma, \SL_2(\CC))$. Orbits of $\SL_2(\CC)$ on $R^*(\Gamma, \SL_2(\CC))$ are closed and so its image $X^*(\Gamma, \SL_2(\CC))$ in $X(\Gamma, \SL_2(\CC))$ is homeomorphic to $R^*(\Gamma, \SL_2(\CC)) / \SL_2(\CC)$. 


\subsection{Character varieties of hyperbolic 3-manifolds} \label{subsec:CharHyp}

Now we restrict to the case of interest to us of studying $X(\pi, \SL_2(\CC))$ where $\pi = \pi_1(M)$ is the fundamental group of an orientable hyperbolic manifold with $\mu$ cusps. We will use the notations
\[
R(M) = R(\pi, \SL_2(\CC)), \quad X(M) = X(\pi, \SL_2(\CC)).
\]
The complete hyperbolic structure of $M$ corresponds to a {\em holonomy representation} $\ovl\rho_0 : \pi \to \PSL_2(\CC)$ such that $M =~\ovl\rho_0(\pi) \backslash \HH^3$. It is known (see \cite{Culler}) that this representation lifts to a representation $\rho_0 \colon \pi_1(M) \to~\SL_2(\mC)$ which we will also call the holonomy representation (there is in general no unique choice for $\rho_0$---we assume that an arbitrary one has been made). It is an irreducible representation and we call its image $[\rho_0]$ in $X(M)$ the holonomy character.

We will need a way to work directly with representations rather than characters, for which the following well-known lemma (see \cite[Proposition 3.2]{Porti_thesis}) will be useful. 

\begin{lemma} \label{section}
  There exists a section of the projection $R(M) \to  X(M)$ which is real-analytic on a neighbourhood of the holonomy character $[\rho_0]$.
\end{lemma}

\section{The $L^2$-torsion} \label{sec:L2}

We introduce the twisted $L^2$-torsion $\todeux$ for unimodular representations of 3-manifold groups. Twisted $L^2$-torsions are considered in a general context in \cite{Lueck_twisted} and we use this and the book \cite{Lueck} as a reference; we recall enough in Subsection \ref{subsec:L2} to be self-contained minus proofs. In Subsection \ref{subsec:twisting} we use this to define $\todeux$. In Subsection \ref{subsec:comp} we give some useful properties and some simple examples. Finally, in Subsection \ref{subsec:Reg} we prove a regularity lemma for the determinant which we will use in Section \ref{sec:FKdet}. 


\subsection{$L^2$-invariants} \label{subsec:L2}

Let $\pi$ be a discrete group and $\CC\pi$ its group ring. The reduced von Neumann algebra $\neu\pi$ is the commutant of the left action
of $\CC\pi$ on $\ell^2(\pi)$ in the algebra of bounded operators on the Hilbert space $\ell^2(\pi)$ (see \cite[Definition 1.1]{Lueck}). A finitely generated Hilbert $\neu\pi$-module is a
quotient of some $\ell^2(\pi)^n$ by a closed $\neu\pi$-invariant subspace \cite[Definition 1.5]{Lueck}. There is a functor $\Lambda$ from the category of based free $\CC\pi$-modules to the category of Hilbert $\neu\pi$-modules \cite[p.~727]{Lueck_twisted} ($\Lambda M$ can be defined as the $L^2$-completion of $M$ with respect to the basis $g\cdot b$ where $b$ is an element of the $\CC\pi$-basis of $M$ and $g \in \pi$). In the sequel we will often make no distinction between $M$ and $\Lambda M$ or $f$ and $\Lambda f$. 

For $V$ a Hilbert $\neu\pi$-module we denote by $\mathcal B_{\neu\pi}(V)$ (resp. $\GL_{\neu\pi}(V)$) the set of bounded (resp. bounded with a bounded inverse) $\neu\pi$-equivariant operators on $V$, and by $\tr_{\neu\pi}$ the usual trace on $\mathcal B_{\neu\pi}(V)$ if $V$ is of finite type (see \cite[Definitions 1.2, 1.8]{Lueck}). For the definition of the $L^2$-torsion we need the definition of the Fuglede--Kadison determinant $\det_{\mathcal N\pi}(A)$ of a morphism $A$ between two Hilbert $\mathcal N\pi$-modules. If $A$ is a positive self-adjoint bounded operator in $\mathcal B_{\neu\pi}(V)$ with trivial kernel then the operator $\log(A)$ is well-defined---but not bounded in general. If $\log(A)$ is in $\mathcal B_{\neu\pi}(V)$ its Fuglede--Kadison determinant of $A$ can be computed by the formula: 
\begin{equation} \label{def_FKdet}
  \detn_{\neu\pi}(A) = \exp(\tr_{\neu\pi}\log(A)).
\end{equation}
The operator $\log(A)$ is in $\mathcal B_{\neu\pi}(V)$ if and only if $A \in \GL_{\neu\pi}(V)$, which is the main case we will consider in the sequel.

The Fuglede--Kadison determinant is defined in general, and the equality above holds formally (see \cite[Definition 3.11]{Lueck}). When the trace is $-\infty$ the determinant is zero. If $\detn_{\neu\pi}(A) > 0$ then the operator $A$ is said to be {\em of determinant class}. Being of determinant class depends only on the behaviour at $0$ of the spectral density function of $A$; the simplest case is when its Novikov--Shubin invariants (see \cite[Definitions 2.8, 2.16]{Lueck}) are positive.

\begin{remark} \label{detclass}
It is a difficult question to know whether a morphism is of determinant class, except when $\pi$ is virtually abelian (if $\pi$ is such then all morphisms between finitely generated $\CC\pi$-modules are of determinant class, as follows from \cite[Example 3.13]{Lueck} and \cite[Lemma 1.8]{Everest_Ward}). 

For many groups $\pi$ (including all residually finite groups, so in particular al 3--manifold groups) is is known that when $f$ is given by a matrix over the integral group ring $\ZZ\pi$ it is of determinant class (see \cite[Chapter 13]{Lueck}). So if $\rho$ is defined over $\QQ$ or even $\overline\QQ$ we have $\todeux \not= 0$. However when $\pi$ is nonabelian and the representation has transcendental coefficients it is essentially completely open to give general criteria. 
\end{remark}

If $(C_*, d_*)$ is a complex of Hilbert $\neu\pi$-modules which is $L^2$-acyclic (its reduced $L^2$-homology \cite[Definition 1.17]{Lueck} vanishes, that is $\mathrm{Im}(d_p)$ is dense in $\ker(d_{p-1})$) then the combinatorial Laplace operators
\begin{equation} \label{defn_comb_lapl}
  \Delta_p = d_pd_{p+1}^* + d_{p+1}d_p^*
\end{equation}
have zero kernel. They are always positive, and the complex $(C_*, d_*)$ is said to be {\em of determinant class} if all $\Delta_p$ are of determinant class. Then the $L^2$-torsion $\tau^{(2)}(C_*) \in ]0, +\infty[$ is defined by
\begin{equation} \label{defn_tors_alg}
  \tau^{(2)}(C_*)^2 = \prod_{p=0}^n \detn_{\neu\pi}(\Delta_p)^{(-1)^p p}. 
\end{equation}
We extend the definition of $L^2$-torsion to all complexes by taking the convention that $\tau^{(2)}(C_*) = 0$ if $C_*$ is not of determinant class. 


\subsection{Twisted $L^2$-torsion for unimodular representations} \label{subsec:twisting}

Let $M$ be a finite CW-complex, $\pi$ its fundamental group and $\rho \colon \pi \to \mathrm{GL}(E)$
a finite-dimensional unimodular representation (we will be interested here in $M$ being a 3--manifold and $E = \CC^2$). There is a twisted chain complex of $\CC(\pi)$-modules 
\[
C_*(\wdt M, \rho) = C_*(\wdt M) \otimes_{\CC} E
\]
where any element $\gamma$ of the group $\pi$ acts by $\gamma \otimes \rho(\gamma)$ on $\CC(\pi) \otimes_{\CC} E$ by the right-regular representation 
on the first factor times $\rho$ acting on the left on the second one (this is denoted by $\eta_E C_*(\wdt M)$ in~\cite{Lueck_twisted}).

If we have chosen a $\CC\pi$-basis for $C_*(\wdt M)$ (that is, an orientation and a lift to the universal cover for each cell) and a basis for $E$ then we get a complex of free based $\CC\pi$-modules. If its completion $\Lambda C_*(\wdt M, \rho)$ (which we will rather denote by $C_*^{(2)}(M, \neu\pi\otimes \rho)$ in the rest of this paper)
is $L^2$-acyclic then we define the twisted $L^2$-torsion of $(M, \rho)$ by: 
\begin{equation} \label{defn_todeux}
  \todeux = \tau^{(2)}(C_*^{(2)}(M, \neu\pi \otimes \rho)). 
\end{equation}
Since the representation $\rho$ is unimodular, it follows immediately from \cite[Theorem 6.7(1), Lemma 3.2(1)]{Lueck_twisted} that \eqref{defn_todeux} does not depend on on the choices of bases for $C_*(\wdt M)$ and $E$.

\medskip

We are interested in the twisted $L^2$-torsion as defined above for 3--manifolds; however at some points we will need the $L^2$-torsion for more general covers. The twisted $L^2$-torsion of a cover $\widehat M \to M$ associated to a surjective morphism $\pi \to \Lambda$ is defined exactly as above, using the chain complexes $C_*(\widehat M)$ in place of $C_*(\wdt M)$ and replacing the group ring and von Neumann algebra $\CC\pi, \neu\pi$ by $\CC  \Lambda, \neu \Lambda$ respectively. In this case we will denote the complex by $C_*^{(2)}(M, \mathcal N \Lambda \otimes \rho)$, and the $L^2$-torsion by $\tau^{(2)}(M, \mathcal N \Lambda\otimes \rho)$.


\subsection{Properties and first examples} \label{subsec:comp}

The following basic fact follows immediately from \cite[Lemma 3.2 (1)]{Lueck_twisted}. 

\begin{lemma} \label{conj_indep}
  Let $\rho, \rho' \colon \pi_1(M) \to \SL_2(\CC)$ be two conjugate representations. Then $\todeux = \tau^{(2)}(M, \rho')$.
\end{lemma}

The next well-known lemma (see for example \cite[Lemma 3.1]{DFL}) will be useful to compute the $L^2$-torsion of Seifert manifolds.

\begin{lemma} \label{decomposition}
  Let
  \[
  C_* = 0 \to \CC\pi^k \xrightarrow{A} \CC \pi^{k+l} \xrightarrow{B} \CC \pi^l \to 0. 
  \]
  Let $A' \in M_{k, k}(\CC\pi)$ and $B' \in M_{l, l}(\CC\pi)$ be the square matrices obtained from the $k$ first lines of $A$ and $l$ last columns of $B$ respectively.
  If $A', B'$ are acyclic and of determinant class then so is $C_*$ and moreover
  \[
  \tau^{(2)}(C_*) = \detn_{\neu\pi}(A')^{-1}\detn_{\neu\pi}(B').
  \]
\end{lemma}

\begin{proof}
  Consider the diagram
  \[
  \begin{CD}E_*\\@VVV\\C_*\\@VVV\\D_*\end{CD} \qquad
  \begin{CD}
    0 @>>> 0        @>>>     \CC\pi^l @>{B'}>> \CC\pi^l @>>> 0 \\
    @.      @VVV               @VVV              @|         @. \\
    0 @>>> \CC\pi^k @>A>> \CC\pi^{k+l} @>B>>    \CC\pi^l @>>> 0 \\
    @.      @|                 @VVV              @VVV       @. \\
    0 @>>> \CC\pi^k @>{A'}>> \CC\pi^k @>>>      0       @>>> 0 \\
  \end{CD}
  \]
  where vertical arrows in the middle are given by natural inclusions and projections. This gives a short exact sequence of complexes $0 \to E_* \to C_* \to D_* \to 0$ and it follows from additivity of $L^2$-torsion \cite[Theorem 3.35(1)]{Lueck} (the terms on the right in this equation vanish because of acyclicity and determinants of inclusions and projections are 1) that $C_*$ is acyclic, of determinant class and 
  \[
  \tau^{(2)}(C_*) = \tau^{(2)}(D_*) \cdot \tau^{(2)}(E_*).
  \]
  Since the torsion of a complex with a single nonzero differential $d$ is equal to $\det_{\neu\pi}(d)^{(-1)^p}$ where $p$ is the degree of $d$, the result follows. 
\end{proof}

In some very special cases the computation of the $L^2$-torsion is immediate. 

\begin{proposition} \label{prop:Torsion}
\begin{enumerate}
\item \label{unitary} Assume that $M$ is hyperbolic. If the representation $\rho \colon \pi_1(M) \to \SU(2) \subset \SL_2(\CC)$ is unitary then
\[
\todeux = \tau^{(2)}(M)^2 = \exp(-\Vol(M)/3\pi)
\]

\item \label{abelian} Assume that $H_1(M)= \ZZ$, and let $\phi \colon \pi_1(M) \to \ZZ$ be a choice of an abelianization map. If $\rho\colon \pi_1(M) \to \SL_2(\CC)$ is reducible, conjugated to 
$$\rho(\gamma) = \bm \lambda^{\phi(\gamma)}& \ast \\ 0 & \lambda^{-\phi(\gamma)}\ema$$
for some $\lambda \in \CC^*$ then 
$$\todeux = \tau^{(2)}(M,\phi)(|\lambda|) \tau^{(2)}(M, \phi)(|\lambda|^{-1})$$
where the abelian-twisted $L^2$-torsion $\tau^{(2)}(M,\phi)$ is defined in \cite{DFL}.
\end{enumerate}
\end{proposition}

\begin{proof}
Point \eqref{unitary} is a consequence of \cite[Theorem 4.1 (4)]{Lueck_twisted}

For \eqref{abelian}, observe that using Lemma \ref{conj_indep} one can assume that $\rho$ has the triangular form above. It induces an exact sequence of $\pi$-modules
\[
0 \to \ell^{2}(\pi) \to \ell^{(2)}(\pi)^2 \to \ell^{(2)}(\pi) \to 0
\]
where the $\pi$-actions are given by the representations $\lambda^{\phi}, \rho, \lambda^{-\phi}$ respectively. Now we can apply \cite[Lemma 3.3]{Lueck_twisted} and conclude since we have the equality $\tau^{(2)}(M,-\phi)(|\lambda|) = \tau^{(2)}(M, \phi)(|\lambda|^{-1})$.
\end{proof}


\subsection{Regularity under spectral gap} \label{subsec:Reg}

Let $V$ be a Hilbert space, $\mathcal B(V)$ the space of bounded linear operators on $V$ and $\mathcal S(V) \subset \mathcal B(V)$ the subspace of self-adjoint operators. By the spectral theorem any self-adjoint operator $T$ in $\mathcal S(V)$ has a spectrum $\sigma(T)$ which is a closed subset of $\RR$. It follows from the spectral theorem and Cauchy--Schwarz inequality that we have
\begin{equation} \label{rayleigh}
  \inf \sigma(T) = \inf_{\|v\| = 1} \frac{\|Tv\|}{\|v\|} = \inf_{v \in V \setminus \{0\}} \frac{\|Tv\|}{\|v\|} = \inf_{v \in V \setminus \{0\}} \frac{\langle Tv, v \rangle}{\|v\|^2}
\end{equation}
We then have the following well-known lemma. 

\begin{lemma} \label{spec_continuity}
  The function $\mathcal S(V) \to \RR$ defined by $T \mapsto \inf \sigma(T)$ is continuous. 
\end{lemma}

\begin{proof}
  Let $T \in \mathcal S(V)$ such that $\inf\sigma(T) = \lambda > -\infty$. Let $\varepsilon > 0$. Then for any $v \in V$ with $\|v\|=1$ we have $\|Tv\| \ge \lambda$. If $\|| S - T \|| < \varepsilon$ then for all such $v$ we have $\| Sv \| \ge \|Tv\| - \|Sv - Tv\| \ge \lambda - \varepsilon$. Together with the first equality in \eqref{rayleigh} this proves that $\sigma$ is lower-continuous. On the other hand there exists a $w \in V$, $\|w\| = 1$ such that $\| Tw \| \le (\lambda + \varepsilon) \|$. For $S$ as above we have $\|Sw\| \le \|Tw\| + \|Sw - Tw\| \le \lambda + 2\varepsilon$ so $\sigma(S) \le \lambda + 2\varepsilon$. With \eqref{rayleigh} this proves that $\sigma$ is upper-continous.

  The case where $\sigma(T) = -\infty$ (which is not interesting for us) is similar and left to the reader. 
\end{proof}

Let $\mathcal S^{>0}(V)$ be the set of self-adjoint operators $A \in \mathcal S(V)$ such that $\inf\sigma(A) > 0$, and let $\mathcal S_{\neu\pi}^{>0}(V) = \mathcal S_{>0}(V) \cap \mathcal B_{\neu\pi}(V)$. By Lemma \ref{spec_continuity}, the set $\mathcal S_{\neu\pi}^{>0}(V)$ is an open subset in $\mathcal S_{\neu\pi}(V)$. 

It is an open question to determine exactly the domain of continuity in $\mathcal B_{\neu\pi}(V)$ of the Fuglede--Kadison determinant. However it is a general principle that it is ``as regular as possible'' on the open subset $\GL_{\neu\pi}(V)$; we will use an instance of this valid in the real-analytic category.

The space $\mathcal B_{\neu\pi}(V)$ is a Banach space for the norm operator (it is a norm-closed subspace of $\mathcal B(V)$). If $U \subset \RR^n$ is an open subset and $E$ a Banach space we say that a function $A : U \to E$ is {\em real-analytic} if it admits an expression as a convergent power series in the neighbourhood of every $x \in U$. 

We will use the habitual notations: for a $n$-tuple of integers $\alpha = (\alpha_1, \ldots, \alpha_n)$ and $x \in \RR^n$, let $| \alpha | = \sum_i \alpha_i$ and $x^\alpha = x_1^{\alpha_1} \cdots x_n^{\alpha_n}$.

The former real-analyticity condition can be conveniently formulated as follows: for every $x \in U$ there exists a sequence $(a_\alpha) \in E^{\NN^n}$ and $r_0 > 0$ such that $\sum_{\alpha} \|a_\alpha\|_E \, r_0^{|\alpha|} < +\infty$ and for any $v \in \RR^n$ such that $|v_i| < r_0$ for $1 \le i \le n$ (that is, $\|v\|_\infty < r_0$), we have $A(x+v) = \sum_{\alpha \in \NN^m} a_\alpha v^\alpha$.

\begin{lemma} \label{analytic_det}
  Let $V$ be a Hilbert-$\neu\pi$-module of finite type, $U$ an open subset of $\RR^n$ and $A : U \to \mathcal S_{\neu\pi}^{>0}(V)$ a real-analytic map. Then the map $x \mapsto \det_{\neu\pi} A(x)$ is a real-analytic function on $U$. 
\end{lemma} 

\begin{proof}
  We see from \eqref{def_FKdet} that we need to prove that the function $x \mapsto \tr_{\neu\pi} \log (A(x))$ is real-analytic. This will follow immediately from the two following points:
  \begin{enumerate}
  \item[i)]  If $\ell$ is a continuous linear form on a Banach space $\mathcal B$ and $x \mapsto T(x)$ a real-analytic map $U \to \mathcal B$ then $x \mapsto \ell(T(x))$ is real-analytic;

  \item[ii)] If $I \subset \RR$, $f$ is a real-analytic function with a power series expansion which converges on $I$, and $\sigma(A(x)) \subset I$ for all $x \in U$ then $x \mapsto f(A(x))$ is real-analytic. (We note that $f = \log$ satisfies this assumption for any $I \subset \interval o0{+\infty}o$.)
  \end{enumerate}
  Let us prove i): write $\sum_{\alpha \in \NN^n} T_\alpha v^\alpha$ a convergent power series  with a positive radius of convergence for $T$,  then by continuity of $\ell$ the series $\sum_{\alpha} \ell(T_{\alpha}) v^\alpha$ converges as well, and it equals the image of the former by the linear map $\ell$ on the domain of convergence.

  Now we prove ii): let $f(t) = \sum_k c_k t^k$ be a power series expression for $f \colon I \subset \RR \to \RR$, let $x \in U$ and let $\sum_{\alpha \in \NN^n} a_\alpha v^\alpha$ be the power series expansion for $A(x+v)$, which converges for $\|v\|_\infty \le r_1$ for some $r_1 > 0$. We denote $M = \sum_{\alpha \in \NN^n} \|a_\alpha\|_{\mathcal B}\, r_1^{|\alpha|}$. 

  Assuming $a_0 = 0$ for convenience, we want to show that the formal series $\sum_{k \ge 0} c_k \left(\sum_{\alpha \in \NN^n} a_\alpha v^\alpha \right)^k$ is convergent on a neighbourhood of $v = 0$. Let $r_0$ be the radius of convergence for $\sum_k c_k t^k$, then for $\|v\|_\infty < r_1r_0/M$ we have that $\sum_{\alpha \in \NN^n} \|a_\alpha\|_{\mathcal B}\, |v^\alpha| < r_0$ by convexity. 

  For those $v$ the series $\sum_k |a_k| \left(\sum_{\alpha \in \NN^n} \|a_\alpha\|_{\mathcal B} \, |v^\alpha| \right)^k$ converges; in other words the power series $\sum_{k \ge 0} a_k \left(\sum_{\alpha \in \NN^n} a_\alpha v^\alpha \right)^k$ is absolutely convergent on a neighbourhood of $0$. 
\end{proof}


\section{Seifert manifolds and JSJ decomposition} \label{sec:Seifert}

In this subsection we prove Theorem \ref{seifert} from the introduction, the statement of which we recall here (we note that a similar statement, with a much simpler proof, also holds for torsions with unitary coefficients) :

\begin{theorem}
  Let $M$ be a compact aspherical 3--manifold and $N_1, \ldots, N_r$ be the hyperbolic components in the JSJ decomposition of $M$. Let $\rho \colon \pi_1(M) \to \SL_2(\mathbb C)$ be a representation such that for any Seifert piece $N\subset M$, the restriction $\rho\vert_{\pi_1(N)}$ is irreducible, and for every hyperbolic piece $C_*(N_i, \neu\pi_1(N_i) \otimes \rho|_{\pi_1(N_i)})$ is determinant-class. Then we have
  \[
  \tau^{(2)}(M ; \rho) = \prod_{i=1}^r \tau^{(2)}(N_i ; \rho).
  \]
  In particular if $M$ is Seifert or more generally a graph manifold then $\tau^{(2)}(M ; \rho) = 1$.
\end{theorem}

\begin{proof}
  Since $M$ is aspherical it is in particular irreducible and we can perform a JSJ-decomposition of $M$. Its JSJ-components are either hyperbolic or Seifert fibered manifolds. The theorem then follows immediately from the following two claims :
  \begin{enumerate}
  \item \label{gluing} If $N_1, N_2$ are compact with toric boundary, $N$ is a gluing of $N_1, N_2$ along a collection of incompressible boundary components then
    \[
    \tau^{(2)}(N, \rho) = \tau^{(2)}(N_1,\rho|_{\pi_1(N_1)})\tau^{(2)}(N_2,\rho|_{\pi_1(N_2)}).
    \]

  \item\label{seifert_only}  For any $\rho \colon \pi_1(N) \to \SL_2(\CC)$, if $N$ is Seifert (compact with toric boundary components) then the complex $C_*^{(2)}(N,  \neu\pi_1(N) \otimes \rho\vert_{\pi_1(N)})$ is $L^2$-acyclic and of determinant class. Moreover, if $\rho$ is an irreducible representation of $\pi_1(N)$ then $\tau^{(2)}(N ; \rho) = 1$. 
  \end{enumerate}
  Let $N = N_1 \cup N_2$ as in the statement of Claim \eqref{gluing} and $T_1, \ldots, T_r$ the boundary tori along which $N_1$ and $N_2$ are glued to each other. Let $\pi = \pi_1(N)$, for $X = N_1, N_2, T_j$ we denote by $C_*^{(2)}(X, \rho)$ the twisted complex of Hilbert $\neu\pi$-modules associated to the preimage of $X$ in the universal cover $\wdt N$ (by incompressibility of the $T_j$ this is a disjoint union of copies of $\wdt X$). By the restriction law \cite[Theorem 6.7(6)]{Lueck_twisted} we have that
  \[
  \tau^{(2)}(C_*^{(2)}(X, \rho)) = \tau^{(2)}(X, \rho|_{\pi_1(X)})
  \]
  for those $X$. As we have an exact sequence
  \[
  0 \to \bigoplus_{j=1}^r C_*^{(2)}(T_j, \rho) \to C_*^{(2)}(N_1, \rho) \oplus C_*^{(2)}(N_2, \rho) \to C_*^{(2)}(N, \rho) \to 0
  \]
  it follows by using the gluing formula \cite[Theorem 6.7(3)]{Lueck_twisted} that
  \[
  \tau^{(2)}(N, \rho) = \frac{\tau^{(2)}(N_1,\rho|_{\pi_1(N_1)}) \cdot \tau^{(2)}(N_2,\rho|_{\pi_1(N_2)})} {\prod_{j=1}^r \tau^{(2)}(T_j, \rho|_{\pi_1(T_j)})}. 
  \]
  In Lemma \ref{torus} below we show that the $L^2$-torsion of a 2-torus with coefficients in any local system is 1 so $\tau^{(2)}(T_j, \rho|_{\pi_1(T_j)}) = 1$ for all $j$ and we finally obtain the formula in Claim \eqref{gluing}. 

Claim \eqref{seifert_only} is a subcase of the more general Proposition \ref{seifert_proof} that we prove below. The arguments are essentially lifted from \cite{Kitano} and adapted to the $L^2$-setting. 
\end{proof}

We now prove a lemma on tori used in the proof above and which we will also need to deal with Seifert manifolds (it follows from general Poincaré duality but we prefer to give a more explicit proof). 

\begin{lemma} \label{torus}
  Let $\rho : \mathbb Z^2 \to \SL(V)$ be any unimodular representation, $\Lambda$ an infinite group and $\ZZ^2 \to \Lambda$ a surjective morphism. Then the complex $C_*^{(2)}(\mathbb T^2, \mathcal N\Lambda \otimes \rho)$ is $L^2$-acyclic and of determinant class, and its torsion $\tau^{(2)}(\mathbb T^2, \mathcal N\Lambda \otimes \rho) = 1$. 
\end{lemma}

\begin{proof}
The determinant class condition is always satisfied (see Remark \ref{detclass}). So we need to compute the torsion and show it equals 1. 

First we reduce to the case where $\rho$ is semisimple. Let $m, \ell$ be generators for $\pi_1(\mathbb T^2)$. As $\rho(m)$ and $\rho(\ell)$ commute they are trigonalizable with Jordan blocks of the same shape. It follows that there exists a sequence $g_n \in \SL(V)$ such that $\rho_n(\ell) := g_n\rho(\ell)g_n^{-1}$ and $\rho_n(m) :=g_n\rho(m)g_n{-1}$ converge to a pair of commuting semisimple elements $\rho_\infty(\ell), \rho_\infty(m) \in \SL(V)$. By Lemma \ref{conj_indep} we have $\tau^{(2)}(\mathbb T^2, \rho) = \tau^{(2)}(\mathbb T^2, \rho_n)$. It follows from the formula for the Fuglede--Kadison determinant over abelian groups that $\rho \mapsto \tau^{(2)}(\mathbb T^2, \rho)$ is continuous. So $\tau^{(2)}(\mathbb T^2, \rho) = \tau^{(2)}(\mathbb T^2, \rho_\infty)$.

We can now assume $\rho$ to be semisimple, so we can decompose $\rho \cong \bigoplus_{j=1}^{\dim(V)} \chi_j$ where $\chi_j$ are 1-dimensional. It follows that $\tau^{(2)}(\mathbb T^2, \rho) = \prod_j \tau^{(2)}(\mathbb T^2, \chi_j)$ so we can just assume that $\rho$ is 1-dimensional.

Now the complex of which we have to compute the $L^2$-torsion is just
\[
0 \to \ell^2(\ZZ^2) \xrightarrow{d_1} \ell^2(\ZZ^2)^2 \xrightarrow{d_0} \ell^2(\ZZ^2) \to 0
\]
with the boundary operators given by the following matrices over the group ring:
\begin{align*}
  d_1 &= \begin{pmatrix} 1- \ell \otimes \rho(\ell) \\ 1-m \otimes \rho(m) \end{pmatrix} \\
  d_0 &= \begin{pmatrix} m \otimes \rho(m) - 1 & \ell \otimes \rho(\ell) - 1 \end{pmatrix}. 
\end{align*}
It is immediate that the complex is $L^2$-acyclic, and it follows from Lemma \ref{decomposition} that the torsion is equal to $\det_{\neu\Lambda}(1- \ell \otimes \rho(\ell))^{-1} \det_{\neu\Lambda}(1- \ell \otimes \rho(\ell)) = 1$. 
\end{proof}


\begin{proposition} \label{seifert_proof}
  Let $N$ be a compact Seifert manifold. Let $\rho$ be an irreducible representation of $\pi_1(N)$. Then the complex $C_*(N, \rho)$ is weakly $L^2$-acyclic and of determinant class, and $\tau^{(2)}(N; \rho) = 1$. 
\end{proposition}

\begin{proof}
  We will denote by $h \in \pi_1(N)$ the class of the generic fiber. Weak acyclicity and determinant class property follow from the explicit computation we will make for the complexes involved. We decompose $N$ as $N_0 \cup \bigcup_{i=1}^r T_i$, where $N_0$ is a trivial $\mathbb S^1$-bundle over a surface $S$ with genus $g$ and $k+r$ boundary components ($k$ is the number of boundary components of $N$), and the $T_i$ are neighborhoods of the exceptional fibers, each of them homeomorphic to a solid torus. We will assume that $S$ is orientable; in particular we have the following presentation for the fundamental group of $N$: 
  \begin{multline}
    \pi_1(N) = \langle a_1, b_1 \ldots, a_g, b_g, c_1, \ldots, c_k, q_1, \ldots q_r, h \\ \mid hx = xh\,  \forall x \in \pi_1(N), q_i^{p_i}h_i^{q_i}=1,  [a_1, b_1] \cdots [a_g, b_g] c_1 \cdots c_k q_1\ldots q_r  =1\rangle
  \end{multline}
  Lemma \ref{torus} yields
  \[
  \tau^{(2)}(\partial T_i, \mathcal N \pi_1(T_i) \otimes \rho|_{\pi_1(\partial T_i)}) = 1.
  \]
  Applying the multiplicativity formula for the $L^2$-torsion we get
  \[
  \tau^{(2)}(N, \rho) = \tau^{(2)}(N_0, \mathcal N\pi_1(N) \otimes \rho|_{\pi_1(N_0)}) \prod_{i=1}^r \tau^{(2)}(N_i, \rho|_{\pi_1(N_i)})
  \]
  First we prove that $ \tau^{(2)}(N_i, \rho|_{\pi_1(N_i)})=1$, and then that $\tau^{(2)}(N_0, \mathcal N\pi_1(N)\otimes \rho|_{\pi_1(N_0)})=1$.

  Since $\rho$ is irreducible  and $h$ is central in $\pi_1(N)$, necessarily $\rho(h) = \pm \mathrm{Id}$, and it follows that for each $i$, the generator $q_i$ of $\pi_1(T_i)$ has finite order through~$\rho$. In particular its eigenvalues $\lambda_i, \lambda^{-1}_i$ are unitary complex numbers, and the operator $1- q_i \otimes \rho(q_i)  \colon \ell^{(2)}(\pi_1(T_i))^2 \to \ell^{(2)}(\pi_1(T_i))^2$ has Fuglede--Kadison determinant equal to $1$. Since $T_i$ retracts on the circle $q_i$, its $L^2$ torsion is $\tau^{(2)}(T_i, \rho_{\pi_1(T_i)}) = \det\nolimits_{\mathcal N\pi_1(T_i)} (1- q_i \otimes \rho(q_i))^{-1} = 1$, and the first assertion is proved.

  Now $N_0$ retracts onto 2-complex $Y$ which is a product of a circle with a bouquet of $2g+k+r-1$ circles indexed by $a_1, b_1 \ldots, a_g,b_g, c_1, \ldots, c_k, q_1, \ldots q_{r-1}$ (the last generator $q_r$ does not appear thanks to the last relation in $\pi_1(N)$).

 By the same computation as in \cite[Proof of Proposition 4.2]{Kitano} the differentials in the $\pi_1(N)$-complex $C_*(Y, \mathcal N\pi_1(N) \otimes \rho|_{\pi_1(Y)})$  are given by
  \[
  A = (a_1 \otimes \rho(a_1) - 1 \cdots b_g \otimes \rho(b_g)- 1 \, c_1 \otimes \rho(c_1)  -1 \ldots q_{r-1} \otimes \rho(q_{r-1}) -1 \,  h \otimes \rho(h) - 1)
  \]
  and
  \[
  B =
  \begin{pmatrix}
    1 - h \otimes \rho(h) & 0      & \cdots &  \\
    0                    & \ddots & \ddots &   \\
    \vdots               & \ddots & \ddots &   \\
    0                    & \cdots &        & 1 - h \otimes \rho(h) \\
    1-a_1 \otimes \rho(a_1)  & 1-b_1 \otimes \rho(b_1) & \cdots & 1 -q_{r-1} \otimes \rho(q_{r-1})
  \end{pmatrix}
  \]
  By Lemma \ref{decomposition} the $L^2$-torsion of the complex
  \[
  0 \to L^2(\pi_1(N))^{2g+k+r-1} \xrightarrow{B} L^2(\pi_1(N))^{2g+k+r} \xrightarrow{A} L^2(\pi_1(N)) \to 0
  \]
  is equal to $\det_{\mathcal N\pi_1(N)}(B') \det_{\mathcal N\pi_1(N)}(A')^{-1}$ where
  \[
  B' =
  \begin{pmatrix}
    1 -  h \otimes \rho(h)  & 0      & 0 &  \\
    0                    & \ddots & 0 &   \\
    0                    & 0 & 1 -  h \otimes \rho(h)  \\
  \end{pmatrix}
  \]
  is the matrix obtained by keeping all but the last line of $B$ and 
  $$A' = (1 - h \otimes \rho(h) )$$ 
  is the matrix obtained by deleting all but the last column of $A$. Both of these Fuglede--Kadison determinants equal 1. So $\tau^{(2)}(Y, N\pi_1(N) \otimes \rho|_{\pi_1(Y)})= 1$, and by homotopy invariance \cite[Theorem 6.7(2)]{Lueck_twisted} it follows that $\tau^{(2)}(N_0, \mathcal N\pi_1(N) \otimes \rho|_{\pi_1(N_0)}) = 1$ as well, and it proves the second assertion and finishes the proof of Proposition \ref{seifert_proof} in this case.

  \medskip

  It remains to deal with the case where $S$ is not orientable; this can be done following a similar scheme as above (see \cite{Kitano} for the classical case). A simpler argument in our setting is to apply the orientable case to the double cover associated to $\pi_1(N) \to \pi_1(S) \to \ZZ/2$ and use multiplicativity of the twisted $L^2$-torsion in covers \cite[Theorem 6.7(5)]{Lueck_twisted}. 
\end{proof}

\begin{remark} \label{remk:abel}
  If $\rho$ is not irreducible then the arguments above are still valid except that the image $\rho(h)$ of the generic fiber can have any complex number as an eigenvalue, hence $\det_{\mathcal N\pi_1(N)}(1 -  h \otimes \rho(h) )$ is not necessarily equal to 1. In general, denoting by $\lambda_\rho$ an eigenvalue of $\rho(h)$ of modulus $\ge 1$, we have
 \begin{equation} \label{reducible}
 \tau^{(2)}(N, \rho) = |\lambda_\rho|^{\sum\limits_{i=1}^r -\frac{q_i}{p_i} + 2g+k+r-2}
 \end{equation}
 where $N$ is a Seifert manifold with base an orientable orbifold of genus g, with $k$ cusps and $r$ singular fibers with singularities $(p_i, q_i)$.

  For example if $M(p, q)$ is the complement of the $(p, q)$-torus knot we have
  \[
  \tau^{(2)}(M(p, q), \rho) = |\lambda_\rho|^{1-\frac 1 p -\frac 1 q}. 
  \] 
  which can also be computed directly by a simpler method, using the presentation $\pi_1(M(p,q)) = \langle a,b \mid a^p=b^q \rangle$.
\end{remark}
\begin{proof}[Proof of Remark \ref{remk:abel}]
With the same matrices as in the proof of Proposition \ref{seifert_proof}, we just have to compute the determinants involved, namely
\begin{align*}
\det\nolimits_{\mathcal N\pi_1(N)} (1-q_i \otimes \rho(q_i)) &= |\lambda_\rho|^{q_i/p_i}, \\
\det\nolimits_{\mathcal N\pi_1(N)} (B') &= |\lambda_\rho|^{2g+k+r}, \\
 \det\nolimits_{\mathcal N\pi_1(N)} (A')&= |\lambda_\rho| 
 \end{align*}
 and the result \eqref{reducible} follows. 
\end{proof}


\section{Twisted $L^2$ torsion for hyperbolic manifolds.} \label{sec:FKdet}

In this section we conclude the proof of Theorem \ref{theo:Analyticity}, whose statement we recall here:

\begin{theorem}
  The twisted $L^2$-torsion function 
  \[
  \todeux \colon X(M) \to \mR
  \]
  is real-analytic on an analytic neighborhood $U$ of any lift $[\rho_0]$ of the holonomy representation of $M$ in the character variety $X(M)$.
\end{theorem}

In Subsection~\ref{subsec:Comb} and \ref{subsec:L2torsion} we give alternative definitions for the $L^2$-torsion  $\todeux$, which allows us to work with operators on a fixed Hilbert space (Lemma \ref{fixedspace}) and for which we have a spectral gap (Lemma \ref{gap}) that allows us to apply Lemma \ref{analytic_det} to deduce Theorem \ref{theo:Analyticity}. We give the proof of Lemma \ref{gap} in Subsection \ref{subsec:Technical}, using comparison with the analytic $L^2$-invariants and the spectral gap property of the holonomy representation established in \cite[Lemma 4.1]{BV}. 


\subsection{Combinatorial Laplacians} \label{subsec:Comb}

\begin{lemma} \label{fixedspace}
  There exists a graduated Hilbert $\neu\pi$-module $V = V_0 \oplus \cdots \oplus V_3$ and functions $D_p : R(M) \to \mathcal \mathrm{Hom}_{\neu\pi}(V_p, V_{p-1})$, $p = 1, \ldots, 3$, which are regular in a Zariski-neighbourhood of $\rho_0$ in $R(M)$ and such that for every $\rho$ the complex $(V_*, D_*(\rho))$ is isomorphic to $C_*^{(2)}(M, \neu\pi \otimes \rho)$.

  In particular, if we set $A_p = D_p^*D_p + D_{p+1}D_{p+1}^*$ then $\rho \mapsto A_p(\rho)$ are analytic in a neighbourhood of $\rho_0$ and we have 
  \begin{equation} \label{defn2_todeux}
  \todeux = \prod_{p=1}^3 \det{}_{\neu\pi}\left(A_p(\rho)\right)^{p(-1)^p}.
  \end{equation}
\end{lemma}

\begin{proof}
  We set
  \[
  L_p = C_p(\wdt M) \otimes_\CC \CC^2
  \]
  where $\pi$ acts by $\gamma \cdot (e \otimes v) = (\gamma e) \otimes v$. We choose an arbitrary $\CC\pi$-base $B_p$ for $C_p(\wdt M)$; then $L_p$ is isomorphic (as a $\CC\pi$-module) to $C_p(\wdt M) \otimes \CC^2$ with diagonal action (as introduced in \ref{subsec:twisting}) via the map $I_p(\rho) : \gamma e \otimes v \mapsto \gamma e \otimes \rho(\gamma)v$ for $e \in B$, $\gamma \in \pi$. Choosing any base of $\CC^2$ we get a basis of $L_p$, and the completion $V_p = \Lambda L_p$ is isomorphic to $C_p^{(2)}(M, \neu\pi \otimes \rho)$ (see also \cite[Lemma 1.1]{Lueck_twisted}). 

  Let $\partial_p$ be the differentials of $C_*(M, \neu\pi \otimes \rho)$, which are given by $\partial_p(e \otimes v) = \partial_p e \otimes v$. The corresponding boundary maps of $V_*$ are given by $D_p(\rho) = I_{p-1}^{-1}(\rho) \partial p I_p(\rho)$. The coefficients of $I_p(\rho)$ in the $\CC\pi$-bases of $L_*, C_*(\wdt M) \otimes \CC^2$
  are rational functions of $\rho$ (they depend only on the coefficients of a finite number of $\rho(\gamma)$), so the first part of the lemma is proven. 
  
  Obviously $A_p = I_p^{-1} \Delta_p I_p$ and \eqref{defn2_todeux} follows.
\end{proof}


\subsection{$L^2$-torsion} \label{subsec:L2torsion}

\subsubsection{$L^2$-cochain complexes} We define in the same way as in \ref{subsec:twisting} the $L^2$-cochain complexes $C_{(2)}^*(M, \neu\pi \otimes \rho)$ as the completion of the complex $C^*(\widetilde M) \otimes \CC^2$ with diagonal action (and similarly for $C_{(2)}^*(\partial M, \neu\pi \otimes \rho)$). The relative $L^2$-cochain complex $C_{(2)}^*(M, \partial M, \neu\pi \otimes \rho)$ is then defined by the exact sequence
\[
0 \to C_{(2)}^*(M, \partial M, \neu\pi \otimes\rho) \to C_{(2)}^*( M, \neu\pi \otimes\rho) \xrightarrow{i^*} C^*_{(2)}(\partial M, \neu\pi \otimes \rho\vert_{\partial M}) \to 0. 
\]
We denote by $\Delta_\rel^p \colon C_{(2)}^p( M, \partial M, \neu\pi \otimes \rho) \to C_{(2)}^p( M, \partial M, \neu\pi \otimes \rho) $ the combinatorial Laplacians of the complex $C_{(2)}^*( M, \partial M, \neu\pi \otimes \rho)$. The crucial point for us is then the following, which we prove in the next subsection.  

\begin{lemma} \label{gap}
  There exists $\lambda_0 > 0$ and a neighborhood $U$ of the holonomy such that for all $[\rho] \in U$, for all $p=0, \ldots 3$, we have $\sigma(\Delta_\rel^p(\rho)) \subset \interval f{\lambda_0}{+\infty}o$. 
\end{lemma}



\subsubsection{Poincar\'e duality} 
We follow the references \cite{Lueck} (see the proofs of Theorems 1.36(3) and 3.93(3)) and \cite[Theorem 6.7(7)]{Lueck_twisted}. There is a homotopy equivalence $P_*\colon C_*(\wdt M) \to~C^{3-*}(\wdt M, \partial\wdt M)$ (the untwisted chain and relative cochain complexes); a construction is given in the proof of Theorem 2.1 in \cite{Wall_surgery_book}. We can extend it by the identity to a homotopy equivalence
$$P_*  \colon C_*^{(2)} (M,\neu\pi \otimes \rho ) \to~C^{3-*}_{(2)}(M, \partial M,\neu\pi \otimes \rho)$$ between the completed twisted complexes. 
By \cite[Theorem 2.19]{Lueck} it follows that they have the same Novikov--Shubin invariants. In particular $\Delta_p(\rho)$ has a spectral gap for any $\rho \in U$ (where $U$ is the neighbourhood of $\rho_0$ given by Lemma \ref{gap}). 


\subsubsection{Proof of Theorem \ref{theo:Analyticity}} The relative cochain complex $C^*_{(2)}(M, \partial M, \neu \pi \otimes \rho)$ is of determinant class for $\rho$ in the neighbourhood $U$ given by Lemma \ref{gap}. By Poincaré duality (see preceding paragraph) we get that $C_*^{(2)}(\wdt M, \rho)$ is of determinant class and that for $\rho \in U$, $\inf\sigma(\Delta_p(\rho)) > 0$, and by Lemma \ref{fixedspace} we get that $(V_*, D_*(\rho))$ are determinant-class and the $A_p(\rho)$ have a spectral gap for $\rho \in U$. 

The hypotheses of Lemma \ref{analytic_det} are thus satisfied by $[\rho] \mapsto A_p(\rho)$ where $\rho$ belongs to the section of the map $R(M) \to X(M)$ given by Lemma \ref{section}. It follows that the functions $[\rho] \mapsto \det_{\neu\pi}A_p(\rho)$ are real-analytic in $U$ and so is the $L^2$-torsion by \eqref{defn2_todeux}.


\subsection{Proof of lemma \ref{gap}} \label{subsec:Technical}

\subsubsection{Preliminaries on analytic $L^2$-invariants}
In this subsection we define the twisted analytic Laplacian operators and derive some of their properties. For our purposes here we need only consider twisting by the holonomy representation $\rho_0$ but the discussion can be adapted to deal with any $\rho$. 

We consider the trivial rank 2 bundle $\HH^3 \times \mC^2$ with the $\pi_1(M)$-action
\[
\gamma \cdot (\widetilde x, v) = (\gamma \cdot \widetilde w, \rho_0(\gamma) \, v)
\]
for any $\gamma \in \pi_1(M), \widetilde x \in \HH^3, v \in \CC^2$, and the associated complex $\Omega_c^*(\HH^3, \CC^2)$ of compactly supported, $\CC^2$-valued forms on $\HH^3$ with the natural $\Gamma$-equivariant differential $d_p \colon \Omega^p(\HH^3, \CC^2) \to \Omega^{p+1}(\HH^3, \CC^2)$. On the bundle $\HH^3 \times \CC^2$ we choose an arbitrary $\SL_2(\CC)$-invariant norm (which amounts to choosing a norm on $\CC^2$). 

The spaces $\Omega_c^p(\HH^3, \CC^2)$ admit a Hermitian product associated with this metric on $\HH^3 \times \CC^2$. We will use a uniform notation for various spaces of differential forms associated with this inner product: if $X$ is a manifold, a prefix $R$ before $\Omega^p(X, \rho)$ indicates that we consider the completion associated to $\rho$ of the $p$-forms with coefficients in the trivial bundle  
and regularity $R$ (for example $L^2\Omega^p$ for square-integrable forms, $H^l\Omega^p$ for Sobolev spaces---note that these Hilbert spaces and their $\neu\pi$-module structure depend on $\rho$). The norm on any $R\Omega^p(X, \rho)$ will be denoted by $\|\cdot\|_R$. For the Sobolev spaces here and after, an integer $l > \dim M =3$ is fixed.

The Hodge Laplacian $\Delta_p$ on the space $L^2\Omega^p(\HH^3, \rho)$ is defined as follows: if $d_p^*$ is the formal adjoint of $d_p$ then the operator defined by $\Delta_p = d_p^*d_p + d_{p+1}d_{p+1}^*$ on the space $\Omega_c^p(\HH^3, \CC^2)$ of compactly supported smooth forms extends to a unique essentialy self-adjoint operator on the completion $L^2\Omega^p(\HH^3, \rho)$. This operator has a well-defined spectrum $\sigma(\Delta_p)$ and it follows from \eqref{rayleigh} then we have
\begin{equation} \label{rayleigh2}
  \inf \sigma(\Delta_p) = \min\left( \inf_{\omega\in \Omega_c^p(\HH^3, \CC^2) \setminus \{0 \}} \frac{\|d_p\omega\|_{L^2}^2}{\|\omega\|_{L^2}^2}, \inf_{\omega\in \Omega_c^{p+1}(\HH^3, \CC^2) \setminus \{0 \}} \frac{\|d_{p+1}\omega\|_{L^2}^2}{\|\omega\|_{L^2}^2} \right). 
\end{equation}
The following result essentially follows from \cite[Lemma 4.1]{BV}; we give a short justification as this reference nominally covers only compact locally symmetric spaces\footnote{The statement of \cite[Lemma 4.1]{BV} is given only for the discrete spectrum but the argument also applies to continuous spectrum as the representations in the packets contributing to the spectrum on $p$-forms also satisfy the condition $\mathrm{Hom}_K(\wedge^p\mathfrak p^* \otimes \CC^2, \cdot) \not= 0$ which is the only point used to give a lower bound on the Casimir eigenvalue. Rigorously justifying this would take too much time so we give a suboptimal bootstrap argument below.}. 

\begin{lemma} \label{gap_complete}
  There exists $\delta > 0$ such that for the holonomy $\rho_0$ and $0 \le p \le 3$ we have $\sigma(\Delta^p(\rho_0)) \subset \interval f\delta{+\infty}o$. 
\end{lemma}

\begin{proof}
  Strictly speaking \cite[Lemma 4.1]{BV} (note that the standard representation of $\SL_2(\CC)$ to itself satisfies the hypothesis therein) applies only to the spectrum of compact hyperbolic 3--manifolds. However, by a standard argument using the trace formula (see for example \cite{deGeorge_Wallach}) the normalised spectral measures of a sequence of compact manifolds $M_n$ with $\mathrm{inj}(M_n) \to +\infty$ converge weakly to the spectral measure for $\HH^3$. The spectral gap persists in the weak limit and the existence of such a sequence $M_n$ (which follows from the residual finiteness of the fundamental group of  hyperbolic 3--manifold) guarantees the spectral gap for $\HH^3$. 
\end{proof}


\subsubsection{Comparison of analytic and combinatorial invariants for the holonomy} \label{preceding}

The main step in the proof of Lemma \ref{gap} is now to prove the spectral gap for $\rho = \rho_0$. We compare the spectra of Laplace operators on $C_{(2)}^*(\wdt M, \partial\wdt M, \neu \pi \otimes \rho)$ and on the Sobolev complexes $H^{l-*}\Omega^*(\HH^3, \rho)$ using Whitney maps. Our arguments are essentially those leading to \cite[Lemma 2.71]{Lueck} which we adapt to the twisted setting. We will use the notation $\|\cdot\|_{\ell^2}$ for the norm on the cochain complex $C_*^{(2)}(\wdt M, \partial\wdt M, \neu\pi \otimes \rho)$, and by $\| \cdot \|_{\wdt x}$ for the norm of the fiber above a point $\wdt x$ in $\HH^3$.

First we recall the definition of Whitney maps. Using the analoguous of Lemma \ref{fixedspace} for cohomology, we use a different (although isomorphic) model for the cochain complex $C^*(M, \neu \pi \otimes \rho)$. We see it as the completion of the complex $ C^*(\wdt M) \otimes \CC^2$ but with action $\gamma\cdot(f\otimes v) = (\gamma f) \otimes (\rho(\gamma)v)$ and differentials $d(f\otimes v) = (df) \otimes v$. If we fix a $\CC\pi$-basis of $C_p(\wdt M)$ the map $\gamma f \otimes v \mapsto \gamma f \otimes \rho(\gamma)v$ gives an isomorphism of $\CC\pi$-complexes from our former model to this one. We choose a smooth partition of unity $e_c$ on $M$, indexed by vertices $c$ of the triangulation and we lift it to $\wdt M$. The important property that we require is that $e_c$ has its support contained in the open simplices adjacent to $c$. Then, if $f_\sigma$ is the cochain with value 1 on $\sigma$ and $0$ on other simplices we define
\[
W^p(f_\sigma \otimes v) = p! \sum_{c \in \sigma} \omega_{\sigma, c} \otimes v
\]
where $\omega_{\sigma, c}$ is defined by the usual formula: if $c_0, \ldots, c_p$ are the vertices of $\sigma$ ordered according to orientation, $\omega_{\sigma, c_i} = (-1)^ie_{c_i} \cdot \bigwedge_{j \not= i} de_{c_j} $. The form $W^pf_\sigma\otimes v$ is supported in the open star of $\sigma$; in particular it is compactly supported in $\wdt M \setminus \partial\wdt M$ if $\sigma$ is not contained in $\wdt \partial M$. 
  
Thus restricting $W$ defines a map $C^*(\wdt M, \partial\wdt M) \otimes \CC^2  \to C_c^\infty\Omega^*(\HH^3, \CC^2)$, and a direct computation (see \cite[p. 140]{Whitney}) shows that it is a chain map---note that in the model we use here, on both sides the differentials satisfy $d(f\otimes v) = df\otimes v$, $d(\omega\otimes v) = d\omega\otimes v$ so we can use this computation as it is for our twisted case. This map from a $\CC\pi$-module of finite rank extends to a bounded chain map
$$W \colon C_{(2)}^*(M, \partial M, \neu \pi \otimes \rho) \to H^{l-*}\Omega^*(\HH^3, \rho).$$
 To prove our claim about the spectral gap of $C_{(2)}^*(\wdt M, \partial\wdt M, \rho)$ it suffices to prove the following inequalities: there exists $C, c > 0$ such that for all $p$ and any cochain $\phi \in C_{(2)}^p(\wdt M, \partial\wdt M, \neu\pi\otimes\rho)$ we have 
\begin{equation} \label{lipschitz}
c\|\phi\|_{\ell^2} \le \|W^p \phi\|_{H^{l-p}} \le C\| \phi \|_{\ell^2}.
\end{equation}
Indeed, if this holds then for any $\phi $ the Rayleigh quotient of $W^p \phi $ satisfies\[
\frac{ \|d(W^p\phi )\|_{H^{l-p-1}} } {\|W^p\phi \|_{H^{l-p}}} = \frac{ \|W^p(d\phi )\|_{H^{l-p-1}} } {\|W^p\phi \|_{H^{l-p}}} \le (c^{-1}C) \frac{\|d\phi \|_{\ell^2}} {\|\phi \|_{\ell^2}}
\]
and by Lemma \ref{gap_complete} and \eqref{rayleigh2} it follows that we must have $\frac{\|d\phi \|_{\ell^2}} {\|\phi \|_{\ell^2}} \ge \frac c C \delta$. As \eqref{rayleigh2} also applies to the combinatorial laplacian we get that $\inf \sigma(\Delta^p_\rel) \ge \frac c C \delta$. This will prove Lemma \ref{gap} for the holonomy representation $\rho_0$, after we show inequality \eqref{lipschitz}.

Since $W$ is bounded, the content of \eqref{lipschitz} is the minoration $c\|\phi \|_{\ell^2} \le \|W^p\phi \|_{H^{l-p}}$, which we will now prove. First we note (see \cite[Section 7]{Schick_notes}) that on the image of $W^p$ the Sobolev and $L^2$-norms are equivalent, so it suffices to show that $c\|\phi \|_{L^2} \le \|W^p\phi \|_{L^2}$. To do so, for a $p$-simplex $\sigma$ in $\widetilde M = \HH^3$ let $\widetilde U_\sigma$ be the maximal open subset of $\HH^3$ on which $\sum_{c \in \sigma} e_c = 1$ and no $e_c$ vanishes for $c \in \sigma$. This is nonempty, and we have $\widetilde U_\sigma \cap \widetilde U_\tau = \emptyset$ if $\tau \not= \sigma$ is another $p$-simplex. In addition, if $\sigma = [c_0, \ldots, c_p]$, replacing every instance of $e_{c_0}$ in $W^p(f _\sigma \otimes v)$ with $(1 - \sum_{i=1}^p e_{c_i}) \otimes v$ (which is valid on $\widetilde U_\sigma$ by definition) we get that
\[
W^p(f_\sigma \otimes v) = (de_{c_1} \wedge \cdots \wedge de_{c_p}) \otimes v \text{ on } \widetilde U_\sigma, 
\]
in particular it does not vanish there and so the integral of its norm on $\widetilde U_\sigma$ is strictly positive. As there are only finitely many $p$-simplexes modulo $\pi$ and the maps $e_{\gamma c}$ are $\pi$-equivariant we get that there exists $a > 0$ such that
\begin{equation} \label{mino_infnorm} 
\| W^p(f_\sigma \otimes v) \|_{L^2(U_\sigma)} \ge a \inf_{\widetilde x \in \widetilde U_\sigma} \|v\|_{\widetilde x}
\end{equation}
for all $\sigma$.

Let $\|\cdot\|$ be the norm on $\CC^2$, and $B$ the $\CC\pi$-basis of $C^*(\wdt M)$, that were used to define the norm $\|\cdot\|_{\ell^2}$ on $C^*_{(2)}(\wdt M, \partial \wdt M, \neu\pi \otimes \rho)$. For any $\sigma$ let $\gamma_\sigma$ the element of $\pi$ such that $\gamma_\sigma^{-1}\sigma$ belongs to $B$. Then by definition $\|f_\sigma \otimes v\|_{\ell^2} = \|\rho(\gamma_\sigma)^{-1} v\|$. As the $\widetilde U_\sigma$ are relatively compact and $B$ is finite, there is some $a'$ (independent of $\sigma$) such that $\inf_{\widetilde x \in \widetilde U_\sigma} \|v\|_{\widetilde x} > a' \|v\|$ for all $\sigma \in B$. As $\widetilde x \mapsto \|v\|_{\widetilde x}$ is $\pi$-invariant, it follows that for any $\sigma$ we have
\[
\inf_{\widetilde x \in \widetilde U_\sigma} \|v\|_{\widetilde x} > a' \|\rho(\gamma^{-1})v\| = a' \|f_\sigma \otimes v\|_{\ell^2}. 
\]
With \eqref{mino_infnorm} we finally get that 
\[
\| W^p(f_\sigma \otimes v) \|_{L^2(U_\sigma)} \ge a'' \|f_\sigma \otimes v\|_{\ell^2}. 
\]
for $a'' = aa' > 0$ and all $\sigma$. Now if $\phi$ is an arbitrary relative cochain we can write it as
\[
\phi = \sum_\sigma  b_\sigma f_\sigma \otimes v_\sigma
\]
so that we have
\[
\| \phi \|_{\ell^2}^2 = \sum_\sigma b_\sigma^2 \|f_\sigma \otimes v_\sigma\|_{\ell^2}^2
\]
As the $\widetilde U_\sigma$ are disjoint we get that 
\begin{align*}
  \| W^p(\phi) \|_{L^2}^2 &\ge \sum_\sigma \| W^p(\phi) \|_{L^2(U_\sigma)}^2 \\
  &= \sum_\sigma \| W^p(f_\sigma \otimes v) \|_{L^2(U_\sigma)}^2 \\
  &\ge (a'')^2 \sum_\sigma b_\sigma^2\|f_\sigma \otimes v\|_{\ell^2}^2 = (a'')^2\|\phi\|_{\ell^2}^2
\end{align*}
where the inequality on the second line follows from the fact that $W^p(f_\tau \otimes v)$ vanishes identically on $\widetilde U_\sigma$ whenever $\tau \neq \sigma$, and that on the third from the previous inequality. The last line is the inequality we were after.

\subsubsection{Conclusion of the proof of Lemma \ref{gap}}

By Lemma \ref{fixedspace} and the isomorphisms $V_* \to C_*^{(2)}(\wdt M, \neu\pi \otimes \rho) \to C_{(2)}^*(\wdt M, \partial\wdt M, \neu\pi \otimes \rho)$ the matrices of the operators $\Delta_p(\rho)$ have coefficients which vary continuously for $[\rho]$ in their domain of definition $D$. In particular each map $\Delta_\rel^p \colon D \to~\mathcal B(C_{(2)}^p(\wdt M, \partial\wdt M, \neu\pi \otimes \rho))$ are continuous for the operator norm. Lemma \ref{gap} then follows from Lemma \ref{spec_continuity} since it holds at $[\rho_0]$ by \ref{preceding} and Lemma \ref{gap_complete}.






\bibliography{biblio}
\bibliographystyle{plain}

\end{document}